\documentclass[a4paper,11pt,reqno,dvips]{amsart}
\usepackage{amsmath,a4,amsfonts,amssymb,amsthm,enumerate,fleqn,geometry}
\usepackage{graphicx}
\usepackage[dvips]{epsfig}
\def\CC{\mathbb{C}}

\def\RR{\mathbb{R}}
\def\NN{\mathbb{N}}

\def\A{\mathcal{A}}
\theoremstyle{plain}

\newtheorem{theorem}{\bf Theorem}[section]

\theoremstyle{remark}
\newtheorem{definition}[theorem]{\bf Definition}

\newtheorem{example}[theorem]{\bf Example}

\newtheorem{corollary}[theorem]{\bf Corollary}

\title[Computable g-frames]{Computable g-frames}

\author[Poonam Mantry]{Poonam Mantry}
\address{Daulat Ram College}
\email{poonam.mantry@gmail.com}

\author[S.K. Kaushik]{S.K. Kaushik}
\address{Kirorimal College}
\email{shikk2003@yahoo.co.in}

\begin{document}

\begin{abstract}
The notion of g-frames for Hilbert spaces was introduced and studied by Wenchang Sun [16] as a generalization of the notion of frames. In this paper, we define computable g-frames in computable Hilbert spaces and obtain computable versions of some of their characterizations and related results.
\end{abstract}

\thispagestyle{empty}
\maketitle

\thispagestyle{empty}

\section{Introduction}
$g$-frames are natural generalizations of frames and share many useful properties with frames. The redundant nature of frames is a desirable property in many practical problems where error tolerance and noise suppression are required. However, a number of new applications have emerged, which cannot be modelled naturally by one simple frame, so some generalizations of frames have been given. In order to make a systematic research on generalized frames, the concept of $g$-frames was introduced by Wenchang Sun [16] in the year 2006.\\
In this paper, we extend the notion of computability of $g$-frames in Hilbert spaces and obtain computable versions of related results. The fact that results on relationship between $g$-frames and frames do not extend directly in the computable version, is supported by examples and then sufficient conditions are derived for the same. The notion of computable dual $g$- frames is also discussed and sufficient conditions for their existence are derived.
\section{ Preliminaries}
In this section, we briefly summarize some notions from computable analysis as presented in [18]. Computable Analysis is the Turing machine based approach to computability in analysis. Pioneering work in this field has been done by Turing[17], Grzegorczyk[9], Lacombe[13], Banach and Mazur[3], Pour-El and Richards[15], Kreitz and Weihrauch[12] and many others. The basic idea of the representation based approach to computable analysis is to represent infinite objects like real numbers, functions or sets by infinite strings over some alphabet $\Sigma$ (which at least contains the symbols 0 and 1). Thus, a representation of a set X is a surjective function $\delta:\subseteq \Sigma^{\omega}\rightarrow X$ where $\Sigma^{\omega}$ denotes the set of infinite sequences over $\Sigma$ and the inclusion symbol indicates that the mapping might be partial. Here, $(X,\delta)$ is called a represented space. \\
A function $F:\subseteq\Sigma^\omega \to \Sigma^\omega$ is said to be computable if there exists some Turing machine, which computes infinitely long and transforms each sequence $p$, written on the input tape, into the corresponding sequence $F(p)$, written on one way output tape.  Between two represented spaces, we define the notion of a computable function.
\begin{definition}$([7])$
Let $(X,\delta)$, $(Y,\delta^\prime)$ be represented spaces. A function $f:\subseteq X\to Y$ is called $(\delta,\delta^\prime)$-computable if there exists a computable function $F:\subseteq\Sigma^\omega \to \Sigma^\omega$ such that $\delta^\prime F(p)= f\delta(p)$, for all $p\in dom(f\delta)$.
\end{definition}
We simply call, a function $f$ computable, if the represented spaces are clear from the context. For comparing two representations $\delta$, $\delta^{'}$of a set $X$, we have the notion of reducibility of representations. $\delta$ is called reducible to $\delta^{'}$, $\delta \leq \delta^{'}$ (in symbols), if there exists a computable function $F:\subseteq \Sigma^{\omega}\rightarrow \Sigma^{\omega}$ such that $\delta(p)= \delta^{'}F(p)$, for all $p\in dom(\delta)$. This is equivalent to the fact that the identity $I:X\rightarrow X$ is $(\delta, \delta^{'})$- computable. If $\delta \leq \delta^{'}$ and $\delta^{'} \leq \delta$, then $\delta$ and $\delta^{'}$ are called computably equivalent.
Analogous to the notion of computability, we can define the notion of $(\delta,\delta^{'})$-continuity, by substituting a continuous function $F$ in the Definition 2.1. On $\Sigma^{\omega}$, we use the Cantor topology, which is simply the product topology of the discrete topology on $\Sigma$.
\\Given a represented space $(X,\delta)$, a computable sequence is defined as a computable function $f:\NN\to X $ where we assume that $\NN $ is represented by $\delta_{\NN}(1^{n}0^{\omega})= n$ and a point $x\in X$ is called computable if there is a constant computable sequence with value $x$. The notion of $(\delta,\delta^{'})$-continuity agrees with the ordinary topological notion of continuity, as long as, we are dealing with admissible representations.
\\ A representation $\delta$ of a topological space $X$ is called admissible if $\delta$ is continuous  and if the identity $I:X\to X$ is $(\delta^\prime,\delta)$-continuous for any continuous representation $\delta^\prime$ of $X$. If $\delta$, $\delta^\prime$, are admissible representation of topological spaces $X$,$Y$, then a function $f:\subseteq X\to Y$ is $(\delta,\delta^\prime)$ continuous if and only if it is sequentially continuous [4].
\\Given two represented spaces $(X,\delta)$, $(Y,\delta^\prime)$, there is a canonical representation $[\delta\rightarrow\delta^\prime]$ of the set of $(\delta,\delta^\prime)$-continuous functions $f:X\to Y $. If $\delta$ and $\delta^\prime$ are admissible representations of sequential topological spaces $X$ and $Y$ respectively, then $[\delta\rightarrow\delta^\prime]$ is actually a representation of the set $C(X,Y)$ of continuous functions $f:X\rightarrow Y$. The function space representation can be characterized by the fact that it admits evaluation and type conversion. See [18] for details.
\\ If $(X,\delta)$, $(Y,\delta^\prime)$ are admissibly represented sequential topological spaces, then, in the following, we will always assume that $C(X,Y)$  is represented by $[\delta\rightarrow\delta^\prime]$ . It follows by evaluation and type conversion that the computable points in $(C(X,Y),[\delta\rightarrow\delta^\prime])$ are just the $(\delta,\delta^\prime)$-computable functions $f:\subseteq X\to Y$ $[18]$. For a represented space $(X,\delta)$, we assume that the set of sequences $X^{\NN}$ is represented by $\delta^{\NN} = [\delta_{\NN}\rightarrow \delta]$. The computable points in $(X^{\NN},\delta^{\NN})$ are just the computable sequences in $(X,\delta)$.
\\The notion of computable metric space was introduced by Lacombe $[14]$. However, we state the following definition given by Brattka $[6]$.
\begin{definition}([6])
 A tuple $(X,d,\alpha)$ is called a \emph{computable metric space} if
\begin{enumerate}
   \item $(X,d)$ is a metric space.
   \item $\alpha:\NN\to X$ is a sequence which is a dense in $X$.
   \item $do(\alpha\times\alpha):\NN^2\to \RR$ is a computable (double) sequence in $\RR$.

\end{enumerate}
\end{definition}
Given a computable metric space $(X,d,\alpha)$, its Cauchy Representation $\delta_{X}:\subseteq \Sigma^{\omega}\rightarrow X$ is defined as\\
 \begin{align*} \delta_{X}(01^{n_{0}+1}01^{n_{1}+1}01^{n_{2}+1}...):=  \lim_{i\rightarrow\infty}\alpha(n_{i}), \end{align*} \\for all $n_{i} \in \NN$ such that $(\alpha(n_{i}))_{i\in \NN}$ converges  and \\
\begin{align*} d(\alpha(n_{i}),\alpha(n_{j})) < 2^{-i}, \ \text{for all} \ j>i.\end{align*} \\
\\In the following, we assume that computable metric spaces are represented by their Cauchy representation. All Cauchy representations are admissible with respect to the corresponding metric topology.
\\An example of a computable metric space is $(\RR,d_{\RR},\alpha_{\RR})$ with the Euclidean metric given by $d_{\RR}(x,y)= \|x-y\|$ and a standard numbering of a dense subset $Q\subseteq \RR$ as $\alpha_{\RR}\langle i,j,k \rangle = (i-j)/(k+1)$. Here, the bijective Cantor pairing function $\langle,\rangle: \NN^{2}\rightarrow \NN$ is defined as $\langle i,j \rangle = j+(i+j)(i+j+1)/2$ and this definition can be extended inductively to finite tuples. It is known that the Cantor pairing function and the projections of its inverse are computable. In the following, we assume that $\RR$ is endowed with the Cauchy representation $\delta_{\RR}$ induced by the computable metric space given above.
\\Brattka gave the following definition of a computable normed linear space.
\begin{definition}([6])
A space $(X,\|\cdot\|,e)$ is called a \emph{computable normed space} if
 \begin{enumerate}
 \item $\|\cdot\|: X\rightarrow \RR$ is a norm on $X$.
 \item The linear span of $e:\NN\rightarrow X$ is dense in $X$.
 \item $(X,d,\alpha_{e})$ with $d(x,y)= \parallel x-y\parallel$ and $\alpha_{e}\langle k,\langle n_{0},...,n_{k}\rangle \rangle= \Sigma_{i=0}^{k}\alpha_{F}(n_{i})e_{i}$ is a computable metric space with Cauchy representation $\delta_{X}$.

\end{enumerate}
\end{definition}
 It was observed that computable normed space is automatically a computable vector space, i.e., the linear operations are all computable. If the underlying space $(X,\|\cdot\|)$ is a Banach space then $(X,\|\cdot\|,e)$ is called a computable Banach space.
 \\We always assume that computable normed spaces are represented by their Cauchy representations, which are admissible with respect to norm topology. Two computable Banach space with the same underlying set are called computably equivalent if the corresponding Cauchy representations are computably equivalent.\\
In this paper, we will discuss operators on computable Hilbert spaces. Brattka gave the following definition of computable Hilbert spaces.
\begin{definition}([5])
A \emph{computable Hilbert space} $(H,\langle \cdot \rangle,e)$ is a separable Hilbert space $(H,\langle.\rangle)$ together with a fundamental sequence $e:\NN \rightarrow H$ such that the induced normed space is a computable normed space.
\end{definition}

A sequence $(f_{i})_{i\in \NN}$ of elements in a Hilbert space $H$ is called a \emph{frame} for $H$ if there exists constants $A, B > 0$ such that
\begin{align*} A\|f\|^{2} \leq \Sigma_{i\in \NN}|\langle f,f_{i}\rangle|^{2} \leq B\|f\|^{2}, \ \text{for all} \ f\in H. \end{align*}
The numbers $A$ and $B$ are called a lower and upper frame bound for the frame $(f_{i})_{i\in \NN}$.
 The \emph{synthesis operator} $T$ given by $T((c_{i})) = \sum_{i=1}^{\infty} c_{i}f_{i}, (c_{i})_{i\in\mathbb{N}}\in l^{2}$ associated with a frame $(f_{i})$ is a bounded operator from $l^{2}$ onto $H$. The adjoint operator $T^{*} : H \rightarrow l^{2}$ given by $T^{*}(f)= (\langle f,f_{i}\rangle)$ is called the \emph{analysis operator}. For other concepts and results related to frames, refer to [8]. For a Hilbert space $H$ and a sequence of Hilbert spaces $\{H_{i}: i\in \NN\}$, a sequence $\{\Lambda_{i}\in L(H, H_{i}):i\in \NN\}$ is called a generalized frame, or simply a \emph{$g$-frame} for $H$ with respect to $\{H_{i}:i\in \NN\}$ if there exists two positive constants $A$ and $B$ such that
\begin{align*} A\|f\|^{2} \leq \Sigma_{i\in \NN}\| \Lambda_{i}f\|^{2} \leq B\|f\|^{2}, \ \text{for all} \ f\in H. \end{align*}
$A$ and $B$ are called a lower and upper $g$-frame bounds, respectively. Here, $L(H,H_{i})$ is the collection of bounded linear operators from $H$ to $H_{i}$, for $i\in \NN$.

\section{Main results}
The notion of g-frames is a generalization of the notion of frames. We begin this section with the definition of computable version of a g-frame.
\begin{definition}
Let $H$ and $\{H_i, i\in\mathbb{N}\}$ be computable Hilbert spaces. A computable sequence $\{\Lambda_i\in L(H,H_i):i\in\mathbb{N}\}$ is called a computable $g$-frame for $H$ with respect to $\{H_i:i\in \mathbb{N}\}$  if there exist  positive constants $A, B $ with $0 < A \le B < \infty$ such that
\begin{align*}
A\|f\|^{2} \leq \Sigma_{i\in \NN}\|\Lambda_i f\|^2 \leq B\|f\|^{2}, \ \text{for all} \ f\in H.
\end{align*}
\end{definition}
\indent By a computable sequence, we mean that for given $i\in\NN$, $\Lambda_i$ is $[\delta_H\rightarrow\delta_{H_i}]$ computable. This can be seen by defining $L^*(H)=\cup_{i\in\NN} L(H,H_i).$ We equip $L^*(H)$ by the representation $\delta^*$ that represents an element $f\in L^*(H)$ by the number $'i'$ such that $f\in L(H,H_i)$ and a name of $f$ as a continuous function from $H$ to $H_i$, i.e.,\\
\\$\delta^*(\langle p,q\rangle)=f \Leftrightarrow  \delta_{\mathbb{N}}(p)=i, f\in L(H,H_i)$ and $[\delta_H\rightarrow\delta_{H_i}](q)=f$.\\
\\A computable sequence $\{\Lambda_i\in L(H,H_i):i\in\NN \}$ can be considered as a computable function
$\phi:\NN  \rightarrow L^*(H)$  given by
$\phi(i) \rightarrow\Lambda_i\in L(H,H_i), i\in \NN.$ That is, given $\delta_{\NN}$ name of $i\in\NN$, we get $[\delta_H\rightarrow\delta_{H_i}]$ name of $\Lambda_i$.\\

A computable frame for a computable Hilbert space $H$ is equivalent to a computable $g$-frame $(\Lambda_{i})$ for $H$ with respect to $\CC$, if $\|\Lambda_{i}\|$ can also be computed in addition to $[\delta_{H}\rightarrow \delta_{C}]$ name of $\Lambda_{i}$, for given $i\in \NN$.
\begin{theorem}
Let $H$ be a computable Hilbert space. Then, the map
\begin{align*}
\phi:\subseteq H^{\NN} &\rightarrow {H^{\prime}}^{\NN}\\
(f_i)&\rightarrow(\Lambda_{f_i})
\end{align*}
with $dom \phi=\{(f_i):(f_i) \ \text{is a frame for} \ H \},$ which maps each frame $(f_i)$ to a $g$-frame for $H$ with respect to $\CC$ is $(\delta_{H}^{\NN}, \delta_{H^{\prime}}^{\NN})$ computable. Also, the map $\phi^{-1}$ is $(\delta_{H^{\prime}}^{\NN}, \delta_{H}^{\NN})$ computable.
\end{theorem}
\begin{proof}
Given a computable frame $(f_i)$ in $H$, there exists a computable function $F :\NN \rightarrow H$ given by $i \rightarrow f_i$, $i\in \NN$. By computable theorem of Frechet Riesz [5], the map $R:H \rightarrow H^{\prime}$ given by $f \rightarrow f_y$, $f\in H$ is $[\delta_H\rightarrow \delta_{H^{\prime}}]$ computable, where $f_y:H\rightarrow F$ is given by $f_y(x)=\langle x,y\rangle, x\in H$. By composition, the map $R\circ F:\NN \rightarrow H^{\prime}$ given by
$i \rightarrow \Lambda_{f_i}$ is computable, where $\Lambda_{f_i}:H\rightarrow \CC$ is given by $\Lambda_{f_i}(f)=\langle f,f_i\rangle$ , $f\in H$. The sequence $(\Lambda_{f_i})$ is a computable $g$-frame for $H$ with respect to $\CC$, with respect to $\delta_{H^{\prime}}$ representation.\\
Now, given a computable $g$-frame $(\Lambda_i)$ for $H$ with respect to $\CC$, with respect to $\delta_{H^{\prime}}$ representation, the composition of the computable maps
$G:\NN \rightarrow H^{\prime}$ given by $i \rightarrow \Lambda_i$, $i\in \NN$ and $R^{-1}:H^{\prime} \rightarrow H$ given by $f_y \rightarrow f$ gives the computability of the map $R^{-1}\circ G:\NN \rightarrow H$ given by $i \rightarrow f_i$ where $(f_i)$ is the sequence in $H$ which satisfies $\Lambda_{f_i}:H\rightarrow \CC$ given by $\Lambda_{f_i}(f)=\langle f,f_i\rangle$, $i\in \NN$. The sequence $(f_i)$ is a computable frame for $H$.
\end{proof}

For a sequence $(H_i)_{i\in\NN}$ of Hilbert spaces, consider
\begin{align*}
\bigg(\sum_{i\in\NN}\oplus H_i\bigg)_{l_2}=\bigg\{(f_i):f_i\in H_i \ \text{and} \ \sum_{i\in\NN}||f_i||^2<\infty\bigg\}
\end{align*}
with inner product defined by $\langle (f_i),(g_i)\rangle$= $\sum_{i\in\NN}\langle f_i,g_i\rangle$. It is known that $\bigg(\sum_{i\in\NN}\oplus H_i\bigg)_{l_2}$ is a Hilbert space with pointwise operations.\\Define $E_{ij}\in\bigg(\sum_{i\in\NN}\oplus H_i\bigg)_{l_2}$ as
\begin{equation*}
(E_{ij})_k =
\begin{cases}
e_{ij},  k=i, \\
0,       k\neq i
\end{cases}
\end{equation*}
where $(e_{ij})_{j\in\NN}$ is an orthonormal basis of $H_{i}$, $i\in \NN$. Then, $(E_{ij})_{i,j\in\NN\times\NN}$ is an orthonormal basis of $\bigg(\sum\oplus H_i\bigg)_{l_2}.$
\\We now assume $H_i$ to be computable Hilbert space, for each $i\in\NN$ and let $(e_{ij})_{j\in\NN}$  be computable orthonormal basis of $H_i, i\in\NN$.\\
It is easy to see that $\bigg\{\bigg(\sum_{i\in\NN}\oplus H_i\bigg)_{l_2}, \langle \cdot\rangle, (E_{ij})_{i,j\in\NN\times\NN}\bigg\}$ is a computable Hilbert space. Accordingly, the Fourier representation of $\bigg(\sum_{i\in\NN}\oplus H_i\bigg)_{l_2}$ is given by
\begin{align*}
\delta_{Fourier}(\langle p,q\rangle)=(f_i)\Leftrightarrow (\delta_{\mathbb{F}}^{\NN})^{\NN}(p)=((\langle f_i,e_{ij}\rangle)_j)_i \ \text{and} \  \delta_{\RR}(q)=\sum_{i=1}^{\infty}\|f_i\|^{2}.
\end{align*}
We now define another representation of $\bigg(\sum_{i\in\NN}\oplus H_i\bigg)_{l_2}$ as follows:
\begin{align*}
\delta_{(\sum_{i\in\NN}\oplus H_i)_{l_2}}(\langle p,q\rangle)=(f_i)\Leftrightarrow p=\langle p_0,p_1,p_2,...\rangle\in \Sigma^{\omega}
\end{align*}
such that
\begin{align*}
[\delta_{H_0},\delta_{H_1},\delta_{H_2},...]\langle p_0,p_1,p_2,...\rangle=(\delta_{H_0}(p_0),\delta_{H_1}(p_1),...)=(f_0,f_1,f_2,...)
\end{align*}
and
\begin{align*}
\delta_{\RR}(q)=\sum_{i=1}^{\infty}\|f_i\|^{2}
\end{align*}
where $\delta_{H_{i}}$ is the Cauchy representation of $H_{i}$, for $i\in \NN.$
We now analyze the relationship between the above defined representations, $\delta_{Fourier}$ and $\delta_{(\sum_{i\in\NN}\oplus H_i)_{l_2}}$ of $\bigg(\sum_{i\in\NN}\oplus H_i\bigg)_{l_2}$. In the following result, we prove that  $\delta_{(\sum_{i\in\NN}\oplus H_i)_{l_2}}$ $\leq$ $\delta_{Fourier}$ representation.

\begin{theorem}
Let $\Bigg(\bigg(\sum_{i\in\NN}\oplus H_i\bigg)_{l_2}, \langle\cdot \rangle, ((E_{ij})_j)_i\Bigg)$ be a computable Hilbert space, where $(H_i)$ is a sequence of computable Hilbert spaces. The representation $\delta_{(\sum_{i\in\NN}\oplus H_i)_{l_2}}$ is reducible to $\delta_{Fourier}$ representation of $(\sum_{i\in\NN}\oplus H_i)_{l_2}$.
\begin{proof}
Given $\delta_{(\sum_{i\in\NN}\oplus H_i)_{l_2}}$ name of $(f_i)$, we can get $\delta_{H_i}$ name of $f_i$, for given $i\in\NN$. If $(e_{ij})_{j\in\NN}$ be a computable orthonormal basis of $H_i, i\in \NN,$ then for given $i\in\NN,$ we get $\delta_{\mathbb{F}}^{\NN}$ name of $(\langle f_i, e_{ij}\rangle)_j$. Thus, we obtain $(\delta_{\mathbb{F}}^{\NN})^{\NN}$ name of $((\langle f_i, e_{ij}\rangle)_j)_i$ and hence $\delta_{Fourier}$ name of $(f_{i})$.
\end{proof}
\end{theorem}
Next, we prove that $\delta_{Fourier}$ $\leq$ $\delta_{(\sum_{i\in\NN}\oplus H_i)_{l_2}}$ under additional conditions.
\begin{theorem}
Let $\Bigg(\bigg(\sum_{i\in\NN}\oplus H_i\bigg)_{l_2}, \langle\cdot \rangle, ((E_{ij})_j)_i\Bigg)$ be a computable Hilbert space such that each $(f_i)\in(\sum_{i\in\NN}\oplus H_i)_{l_2}$ has a computable sequence of norms $(\|f_i\|)$. Then, $\delta_{Fourier}$ representation is reducible to $\delta_{(\sum_{i\in\NN}\oplus H_i)_{l_2}}$ representation.
\end{theorem}
\begin{proof}
This can be easily seen as $\delta_{Fourier}$ name of $(f_i)$ gives $\delta_{\mathbb{F}}^{\NN}$ name of $(\langle f_i,e_{ij}\rangle)_j$ for given $i\in\NN$. As $(\|f_i\|)$ is a computable sequence, we get $\delta_{H_i}$ name of $f_i$ for given $i\in\NN$. This gives $\delta_{(\sum_{i\in\NN}\oplus H_i)_{l_2}}$ name of $(f_i)$.
 \end{proof}
 If $(e_{ij})_{j\in\NN}$ be an orthonormal basis for $H_i$, for $i\in\NN$, then $\{\Lambda_i:i\in\NN\}$ is a $g$-frame for $H$ with respect to $\{H_{i}: i\in \NN\}$ if and only if $\{\Lambda_i^{*}e_{ij}:i,j\in\NN\}$ is a frame for $H$ by a result in [16]. We analyze the computable version of above result in the following.\\
The example given below shows that if $(\Lambda_i)$ be a computable $g$-frame for $H$, then $\{\Lambda_i^{*}e_{ij}:i,j\in\NN\}$ need not be a computable frame for $H$.
\begin{example}
Consider the computable bijective linear operator on $l^2$ given by $U:l^2\rightarrow l^2$ as\\
 \[\left(\begin{array}{ccccc}
 1& a_{1}& a_{2} & a_{3}& \cdots\\
 0& 1 & a_{1} & a_{2} & \cdots\\
 0 & 0 & 1 & a_{1} & \cdots\\
 \vdots &\vdots &\vdots & \ddots\end{array}\right)\]

where $(a_i)$ is a computable sequence of positive real numbers such that $\|(a_i)\|_{l^2}<2$ exists but is not computable. Since every bounded invertible linear operator on $H$ forms a $g$-frame for $H$, the operator $U$ forms a computable $g$-frame for $l^2$, but $\{U^*\delta_i:i=1,...,\infty\}, (\delta_i)$ being the standard computable orthonormal basis of $l^2$, is not a computable frame for $l^2$ as $U^*\delta_0=(1,a_1,a_2,...)$ is not computable in $l^2$.
\end{example}

We now give a sufficient condition for the computability of the frame $\{\Lambda_i^* e_{ij}:i,j\in\NN\}$ for $H$, given a computable g-frame $\{\Lambda_i:i\in\NN\}$ for $H$ with respect to $\{H_{i}:i\in \NN\}$.
\begin{theorem}
Let $H$ be a computable Hilbert space, $(H_i)$ be a sequence of computable Hilbert spaces with $(e_{ij})_j$ being the computable orthonormal basis of $H_{i}$, $i\in \NN$. If $\{\Lambda_i:i\in\NN\}$ be a computable $g$-frame for $H$ with respect to $\{H_i:i\in \NN\}$ such that $\{\|\Lambda_i^{*}e_{ij}\|:i,j\in\NN\}$ is a computable double sequence, then $\{\Lambda_i^{*}e_{ij}:i,j\in\NN\}$ is a computable frame for $H$.
\end{theorem}
\begin{proof}
For $i,j\in \NN$, consider the map $g_{ij}:H\rightarrow \mathbb{F}$ given by $g_{ij}(f)=\langle \Lambda_{i}f,e_{ij}\rangle$, for $f\in H$. Since $e_{ij}$ is a computable element of $H_{i}$ for given $i,j \in \NN$, therefore, using the computability of the sequence $(\Lambda_{i})$ and the evaluation property, we get that the map
\begin{align*}
(i,j,f) \rightarrow \langle \Lambda_{i}f, e_{ij}\rangle, i,j\in \NN, f\in H
\end{align*}
is computable. By Type conversion, the map $\Phi: \NN \times \NN \rightarrow \mathbb{F}^{H}$ given by $(i,j)\rightarrow g_{ij}$ is computable. Since $\|g_{ij}\|= \|\Lambda_{i}^{*}e_{ij}\|$, $i,j\in \NN$, we get that $g_{ij}$ is $\delta_{H^{\prime}}$ computable for given $i,j\in\NN$. Now by computable Frechet Riesz Theorem [5], we get $\Lambda_i^{*}e_{ij}$ is $\delta_H$ computable for given $i,j\in\NN.$ Hence, $\{\Lambda_i^{*}e_{ij}:i,j\in \NN\}$ is a computable frame for $H$.
\end{proof}
In Example 3.5, one can observe that $(U\delta_i)$ is a computable frame for $l^2$ but $U^{*}$ is not computable. This shows that if $\{\Lambda_i^{*}e_{ik}:i,k\in \NN\}$ is a computable frame, then $\{\Lambda_i:i\in \NN\}$ need not be a computable $g$-frame for $H$. We now obtain a sufficient condition for the same.
\begin{theorem}
Let $H$ be a computable Hilbert space and $(H_i)_{i\in \NN}$ be a sequence of computable Hilbert spaces. If $\{\Lambda_i^{*}e_{ik}:i,k\in \NN\}$ is a computable frame for $H$ such that for each computable $f\in H$, $((\langle f,\Lambda_i^{*}e_{ik}\rangle)_k)_{i}$ is a computable sequence in $l^{2}$, then $\{\Lambda_i:i\in\NN\}$ is a computable $g$-frame for $H$ with respect to $\{H_i:i\in\NN\}$.
\end{theorem}
\begin{proof}
Since $\{\Lambda_i^{*}e_{ik}:i,k\in \NN\}$ is a computable double sequence in $H$ and $\langle \cdot \rangle$ is computable, there exists a computable map
$\phi:\NN\times\NN\times H \rightarrow \mathbb{F}$ given by $(i,k,f) \rightarrow\langle f,\Lambda_i^{*}e_{ik}\rangle$. By type conversion, the map $\phi^{\prime}:\NN\times H \rightarrow \mathbb{F}^{\NN}$ given by $(i,f) \rightarrow (\langle f,\Lambda_i^{*}e_{ik}\rangle)_k$ is a computable map. This gives $\delta_{\mathbb{F}}^{\NN}$ name of $(\langle \Lambda_i f,e_{ik}\rangle)_k$ for given $i\in\NN$ and $f\in H$. Also, for given $i\in\NN$ and $f\in H$, $(\langle f, \Lambda_{i}^{*}e_{ik} \rangle)_k$ is a computable element of $l^2$ and so, $\|(\langle f,\Lambda_i^{*}e_{ik} \rangle)_k\|_{l^2}=\sum_{k=1}^{\infty}|\langle f,\Lambda_i^{*}e_{ik} \rangle|^2=\|\Lambda_i f\|^2$ is $\delta_{\RR}$-computable. Thus, given $i\in\NN$, $\Lambda_i$ is $[\delta_H\rightarrow \delta_{H_i}]$ computable. As $\{\Lambda_i^{*}e_{ik}:i,k\in \NN\}$ is a frame for $H, \{\Lambda_i:i\in\NN\}$ is a $g$-frame for $H$. Hence, $\{\Lambda_i:i\in \NN\}$ is a computable $g$-frame for $H$ with respect to $\{H_i:i\in\NN\}.$
\end{proof}
\textbf{Remark.}
One can observe that in Example 3.5, $(\langle f,Ue_{k}\rangle)_{k}$ is not computable in $l^2$ for computable $f= e_{0}$.\\

We now state a generalized version of Theorem 3.6 wherein the computable orthonormal basis $\{e_{ik}:k\in\NN\}$ of $H_{i}$ is replaced by  a computable frame in $H_{i}$ for each $i\in \NN$.

\begin{theorem}
Let $\{\Lambda_i:i\in\NN\}$ be a computable $g$-frame for $H$ with respect to $\{H_i:i\in\NN\}$ and $\{g_{ij}:j\in \NN\}$ be a computable frame for $H_i$ with bounds $A_i$ and $B_i$, for all $i\in I$ such that $0<A=inf A_i \ \text{and} \ B=sup B_i<\infty$. Then, $\{\Lambda_i^{*}g_{ij}:i,j\in \NN\}$ is a computable frame for $H$ provided $\{\|\Lambda_i^{*}g_{ij}\|:i,j\in \NN\}$ is a computable double sequence.
\end{theorem}
\begin{proof}
$\{\Lambda_i^{*}g_{ij}:i,j\in \NN\}$ is a frame for $H$ by Theorem 3.4 [1]. The computability of the frame $\{\Lambda_i^{*}g_{ij}:i,j\in \NN\}$ follows by proving that the function $G_{ij}: H\rightarrow F$ given by $G_{ij}(f)= \langle f, \Lambda_{i}^{*}g_{ij}\rangle$ for $i,j\in \NN$ is $\delta_{H^{'}}$ computable and then using the converse of Frechet Riesz Theorem[5]. The proof is on similar lines as that of Theorem 3.6.
\end{proof}

Now, the following example shows that Theorem 3.7 cannot be generalized simply by replacing the computable orthonormal basis $\{e_{ik}:k\in\NN\}$ of $H_{i}$ by a computable frame in $H_{i}$, $i\in \NN$.
\begin{example}
Consider $U:l^2\rightarrow l^2$ given by\\
 \[\left(\begin{array}{ccccc}
 1&0 &0 & 0 & \cdots\\
 a_{1}& 1 & 0 & 0 & \cdots\\
 a_{2} & 0 & 1 & 0 & \cdots\\
 \vdots &\vdots &\vdots & \ddots\end{array}\right)\]

where $(a_i)$ is a computable sequence of positive real numbers such that $\|(a_i)\|_{l^2} < 2 $ exists but is not computable. Then $U$ is a non computable $g$-frame for $l^{2}$. The frame $(g_{i})=((1,0,0...),(-a_{1},1,0...),(-a_{2},0,1,0...)...)$ is a computable frame for $l^{2}$. The frame $(U^{*}g_{i})$=
$((1,0,0..),(0,1,0,..),(0,0,1..)...)$ is again a computable frame for $l^{2}$ such that $(\langle f,U^{*}g_{i}\rangle)_{i}$ = $(\langle f,e_{i}\rangle)_{i}$ is computable in $l^{2}$ for each computable $f\in H$.
\end{example}

We now prove the generalization of Theorem 3.7 with some additional assumption.
\begin{theorem}
Let $(H_i)$ be a sequence of computable Hilbert spaces. Let $\{g_{ij}:j \in \NN\}$ be a computable frame for $H_i$ with computable frame operator $S_{i}$, bounds $A_i, B_i,$ for all $i\in\NN$ such that $0<A=inf A_i \ \text{and} \ B=sup B_i<\infty.$ Let $\{\Lambda_i^{*}g_{ij}:i,j \in \NN\}$ be a computable frame for a computable Hilbert space $H$ such that for each computable $f\in H$, $((\langle f,\Lambda_i^{*}g_{ij}\rangle)_j)_{i}$ is a computable sequence in $l^{2}$. Then, $\{\Lambda_i:i\in\NN\}$ is a computable $g$-frame for $H$ with respect to $\{H_i:i\in\NN\}$.
\end{theorem}
\begin{proof}
The sequence $(\Lambda_{i})$ is a $g$-frame by Theorem 3.4 [1]. Since the synthesis operator corresponding to the computable frame $\{g_{ij}:j\in \NN\}$ is computable by Theorem 3.3 [10], for given $i\in\NN$, $\sum_{j=1}^{\infty}\langle \Lambda_{i}f,g_{ij}\rangle g_{ij}$ is $\delta_{H_i}$ computable.
By Frame representation corresponding to the frame $\{g_{ij}:j\in \NN\},$
\begin{align*}
\Lambda_i f=\sum_{j=1}^{\infty}\langle \Lambda_if,g_{ij}\rangle S^{-1}_{i}g_{ij}=S^{-1}_{i}\bigg(\sum_{j=1}^{\infty}\langle \Lambda_{i} f,g_{ij}\rangle g_{ij}\bigg), \ \text{for all} \ i\in\NN, f\in H.
\end{align*}
By computable Banach Inverse Mapping Theorem [6], $S^{-1}_{i}$ is computable and therefore, $\Lambda_i f$ is $\delta_{H_i}$ computable for given $i\in \NN$. Thus, $\{\Lambda_i\}$ is a computable $g$-frame for $H$ with respect to $\{H_i:i\in\NN\}$.
\end{proof}

\textbf{Remark.}
One can observe that the frame operator of the frame $(g_{i})$ in Example 3.9, is not computable as $S(e_{0})= (1+a_{1}^{2}+a_{2}^{2}+...,-a_{1},-a_{2}...)$ where $1+a_{1}^{2}+a_{2}^{2}+...$= $\|(a_i)\|_{l^2}^{2}$ is not $\delta_{\RR}$-computable.\\

We now analyze the synthesis operator, $T:(\sum_{i\in\NN}\oplus H_i)_{l_2}\rightarrow H$ given by $T((f_i))=\sum_{i=1}^{\infty}\Lambda_i^{*}f_i$ associated to a computable $g$-frame $(\Lambda_i)$. In Example 3.5, $U$ is a computable $g$-frame for $l^2$. Its synthesis operator is given by $T((f_i))=U^{*}(f_i)$, where $U^{*}$ is given by
 \[\left(\begin{array}{ccccc}
 1& 0 & 0 & \cdots\\
 a_1 & 1 & 0 & \cdots\\
 a_2 & a_1 & 1 & \cdots\\
 \vdots &\vdots &\vdots & \ddots\end{array}\right).\]
 Clearly, $T$ is not computable. Therefore, given a computable $g$-frame, the associated synthesis operator need not be computable, unlike the result for computable frames. Therefore, it is natural to find a sufficient condition for the same. The following result gives such a sufficient condition.
 \begin{theorem}
 Let $H$ be a computable Hilbert space, $(H_i)_{i\in\NN}$ be a sequence of computable Hilbert spaces. If a $g$-frame $\{\Lambda_i\in L(H,H_i):i\in\NN\}$ is computable such that $\{\|\Lambda_i^{*}e_{ik}\in\NN\|\}$ is a computable double sequence, $(e_{ik})_k$ being the computable orthonormal basis of $H_i$ for $i\in \NN$, then the synthesis operator $T:(\sum_{i\in\NN}\oplus H_i)_{l_2}\rightarrow H$ given by $T((f_i))=\sum_{i=1}^{\infty}\Lambda_{i}^{*}f_i$ is computable.
 \end{theorem}
 \begin{proof}
 $T$ is a bounded operator as $\{\Lambda_i\in L(H,H_i):i\in\NN\}$ is a $g$-frame for $H$. The computability of $T$ follows from the fact that $T(E_{ij})=\Lambda_i^{*}(e_{ij})$ and $\{\Lambda_i^{*}e_{ij}:i,j\in \NN\}$ is a computable double sequence by Theorem 3.6.
 \end{proof}
 Recall that in the case of computable frames, computability of synthesis operator corresponding to a frame implies the computability of the frame [10]. But, this is not the case for $g$-frames as shown by the following example.

 \begin{example}
 Consider the bounded linear operator  $U:l^2\rightarrow l^2$ given by
\[\left(\begin{array}{ccccc}
 1& 0 & 0 & \cdots\\
 a_1 & 1 & 0 & \cdots\\
 a_2 & a_1 & 1 & \cdots\\
 \vdots &\vdots &\vdots & \ddots\end{array}\right)\]
 $(a_i)$ being a computable sequence in $F^{\NN}$ such that $\|(a_i)\|_{l^2}$ exist but is not computable. $U$ is a non-computable $g$-frame, whereas its synthesis operator $T:l^2\rightarrow l^{2}$ given by $T((c_i))=U^{*}((c_i))$ is a computable operator.
 \end{example}

 The following result gives a sufficient condition for the computability of a $g$-frame.
\begin{theorem}
 If the synthesis operator $T:(\sum_{i\in\NN}\oplus H_i)_{l_2}\rightarrow H$ given by $T((f_i))=\sum_{i=1}^{\infty}\Lambda_i^{*}f_i$ for a $g$-frame $\{\Lambda_i\in L(H,H_i):i\in\NN\}$ is a computable surjective operator such that for each computable $f\in H, ((\langle f, T(E_{ij})\rangle)_j)_i$ is a computable sequence in $l^2$. Then, $\{\Lambda_i:i\in\NN\}$ is a computable $g$-frame for $H$ with respect to $\{H_i:i\in\NN\}.$
\end{theorem}
\begin{proof}
Since $T$ is a computable operator, $\{\Lambda_i^{*}e_{ij}:i,j\in \NN\}=\{T(E_{ij}):i,j\in \NN\}$ is a computable double sequence in $H$. As $\{\Lambda_i:i\in \NN\}$ is a $g$-frame for $H$ with respect to $\{H_i:i\in\NN\}, \{\Lambda_i^{*}e_{ij}:i,j\in \NN\}$ is a computable frame for $H$ such that $((\langle f, \Lambda_i^{*}e_{ij}\rangle)_j)_i=((\langle f, T(E_{ij})\rangle)_j)_i$ is a computable sequence in $l^2$ for given $f\in H$. Now, the result follows from Theorem 3.7.
\end{proof}

\textbf{Remark.}
One may observe that in Example 3.12,
\begin{align*}
(\langle f, T(e_{j})\rangle)_j=(\langle f, U^{*}(e_{j})\rangle)_j=(\langle Uf, e_{j}\rangle)_j=Uf
\end{align*}
is not computable in $l^2$, if we take the computable element $f=e_0$.\\

We now analyze the computability of the analysis operator $T^{*}:H\rightarrow (\sum\oplus H_i)_{l^2}$ given by $T^{*}(f)=(\Lambda_i f)$, associated to a computable $g$-frame. In the case of computable frames, the associated analysis operator need not be computable. Since $g$-frames are a generalization of frames, we expect similar situation in the case of $g$-frames. This can be verified by defining the $g$-frame $\Lambda_i:H\rightarrow \CC$, given by $\Lambda_i(f)=\langle f, f_i\rangle, i\in\NN$, where $(f_i)$ is the computable frame in Example 3.10 [11].\\
\\We now give a sufficient condition for the computability of the analysis operator.

\begin{theorem}
Let $H$ be a computable Hilbert space, $(H_i)$ be a sequence of computable Hilbert space with computable orthonormal basis $\{e_{ik}:k\in\NN\}$ for each $i\in\NN$. Let $\{\Lambda_i\in L(H,H_i):i\in \NN\}$ be a computable $g$-frame for $H$. The analysis operator $T^{*}:H\rightarrow (\sum\oplus H_i)_{l^2}$ given by $T^{*}(f)=(\Lambda_i f)_{i=1}^{\infty}$ is computable, provided the corresponding frame $\{\Lambda_i^{*}e_{ik}:i,k\in \NN\}$ has computable analysis operator.
\end{theorem}
\begin{proof}
Since $\{\Lambda_i\in L(H,H_i):i\in \NN\}$ is a computable sequence, for given computable $f\in H$ and $i\in\NN$, $\Lambda_if$ is $\delta_{H_i}$ computable by Evaluation property. Also, as the analysis operator associated to the frame $\{\Lambda_i^{*}e_{ik}:i,k\in \NN\}$ is computable, we get that $\|((\langle f,\Lambda_i^{*}e_{ik}\rangle)_k)_i\|_{l^2}$ is $\delta_{\RR}$-computable for given computable $f\in H$.\\
Since $\|((\langle f,\Lambda_i^{*}e_{ik}\rangle)_k)_i\|_{l^2}=\sum_{i=1}^{\infty}\sum_{k=1}^{\infty}|\langle f,\Lambda_i^{*}e_{ik}\rangle|^2=\sum_{i=1}^{\infty}\|\Lambda_i f\|^2$, we get $\sum_{i=1}^{\infty}\|\Lambda_i f\|^2$ is $\delta_{\RR}$-computable. Thus, we get $\delta_{(\sum\oplus H_i)_{l^2}}$ name of $(\Lambda_i f)_i$ for given computable $f\in H$ and hence $\delta_{Fourier}$ name of $(\Lambda_i f)$ by Theorem 3.3.
\end{proof}

\begin{corollary}
Let $\{\Lambda_i \in L(H,H_i):i\in \NN\}$ be a computable $g$-frame for $H$ with corresponding frame $\{\Lambda_{i}^{*}e_{ik}:i,k\in \NN\}$ being computable with computable analysis operator, then the $g$-frame operator $S_g:H \rightarrow H$ given by $S_g (f)=\sum_{i\in\NN} \Lambda_{i}^{*}\Lambda_{i} f$ is a computable isomorphism.
\end{corollary}
\begin{proof}
The analysis operator $T^{*}$ associated with the $g$-frame is computable by Theorem 3.14 and the synthesis operator $T$ associated with the $g$-frame is computable by Theorem 3.11. This gives the computability of the operator $S_g$. By computable Banach Inverse Mapping Theorem, $S_{g}^{-1}$ is computable as well.
\end{proof}

\begin{corollary}
Let $\{\Lambda_i\in L(H,H_i):i\in \NN\}$ be a computable $g$-frame for $H$ with the corresponding frame $\{\Lambda_i^{*}e_{ik}:i,k\in \NN\}$ being computable with computable analysis operator, then the pseudo-inverse operator $T^{+}:H\rightarrow (\sum\oplus H_i)_{l^2}$ given by $T^{+}(f)=(\Lambda_i S_g^{-1}f)_i$ is computable.
\end{corollary}
\begin{proof}
Let $f\in H$ be computable with respect to Cauchy representation. Then, $S_g^{-1}f$ is computable by Corollary 3.15. Since analysis operator $T^{*}$ is computable by Theorem 3.14, $T^{*}(S_g^{-1}f)=(\Lambda_i S_g^{-1}f)_i$ is a computable element of $(\sum\oplus H_i)_{l^{2}}$.
\end{proof}

We now develop the notion of dual $g$-frames in the context of computability.
\begin{definition}
Let $H$ be a computable Hilbert space, $(H_i)$ be a sequence of computable Hilbert spaces and  $\{\Lambda_i\in L(H,H_i):i\in \NN\}$ be a computable $g$-frame for $H$. A computable $g$-frame $\{\Theta_i\in L(H,H_i):i\in \NN\}$ is called a computable dual $g$-frame of $\{\Lambda_i\}$ if it satisfies
\begin{align*}
f=\sum_{i=1}^{\infty}\Lambda_{i}^{*}\Theta_if, \ \text{for all} \ f\in H.
\end{align*}
\end{definition}
If $\{\Lambda_i\in L(H,H_i):i\in \NN\}$ be a computable $g$-frame such that the corresponding frame $\{\Lambda_i^{*}e_{ik}:i,k\in \NN\}$ is computable with computable analysis operator, then $(\tilde{\Lambda_i})=(\Lambda_i S^{-1}_{g})$ is computable $g$-frame such that
\begin{align*}
f=\sum_{i=1}^{\infty}\Lambda_{i}^{*}\tilde{\Lambda_i}f=\sum_{i=1}^{\infty} {\tilde{\Lambda_i}}^{*}\Lambda_i f, \ \text{for all} \ f\in H.
\end{align*}
Thus, we get a computable canonical dual $g$-frame of $\{\Lambda_i\in L(H,H_i):i\in \NN\}$. Similar to the case of frames, computable $g$-frames need not always have computable dual $g$-frames as can be easily seen by the following example.

\begin{example}
Consider the computable $g$-frame $T:l^{2}\rightarrow l^{2}$ given by
\[\left(\begin{array}{ccccc}
 1& a_1 & a_2 & a_3 & \cdots\\
 0 & 1 & a_1 & a_2 & \cdots\\
 0 & 0 & 1 & a_1 & \cdots\\
 \vdots &\vdots &\vdots & \ddots\end{array}\right),\]
where $(a_i)$ is a computable sequence in $F^{\NN}$ such that $\|(a_i)\|_{l^2}$ exists but not computable. The corresponding frame for the above $g$-frame is given by
\begin{align*}
\{T^{*}(\delta_i):i\in \NN\}=\{(1,a_1,a_2,...), (0,1,a_1,a_2,...),(0,0,1,a_1,...),...\}.
\end{align*}
A dual frame for $\{T^{*}(\delta_i):i\in \NN\}$ is given by
\begin{align*}
\{u_i^{\prime}:i\in \NN\}&=\{(1,0,0,...),(-a_1,1,0,...),(-a_2,-a_1,1,...),...\}\\
&=\{(-a_i,-a_{i-1},...,-a_1,1,0,...);i=0,1,2,...\}.
\end{align*}
Define $\tau:l^2\rightarrow l^2$ as $\tau(f)=\sum_{i=1}^{\infty}\langle f,u_i^{\prime}\rangle e_i$.
Then, $\{\tau\}$ is a dual $g$-frame for $\{T\}$ but is not computable as $\tau(e_0)=(1,-a_1,-a_2,...)$ is not a computable element of $l^2$.
\end{example}

We now give a computable version of a sufficient condition for a $g$-frame to be a dual $g$-frame of a given $g$-frame $(\Lambda_i)$.
\begin{theorem}
Let $H$ be a computable Hilbert space, $(H_i)$ be a sequence of computable Hilbert space and $\{\Lambda_i\in L(H,H_i):i\in\NN\}$ be a computable $g$-frame for $H$ with respect to $\{H_i:i\in\NN\}$. Let $\{\Gamma_i\in L(H,H_i):i\in\NN\}$ be $g$-frame such that
\begin{align*}
\Gamma_j^{*}(e_{jk})=\phi(E_{jk}), j,k\in \NN,
\end{align*}
where $\phi:(\sum\oplus H_i)_{l_{2}}\rightarrow H$ is a computable left inverse of $T^{*}$ such that $\phi^{*}$ is $(\delta_H,
\delta_{(\sum\oplus H_i)_{l_{2}}})$ computable. Then, $(\Gamma_i)$ is a computable dual $g$-frame to $(\Lambda_i)$.
\end{theorem}
\begin{proof}
$(\Gamma_i)$ is a dual $g$-frame to $(\Lambda_{i})$ by Theorem 3.3 [2]. Since $(\Gamma_i)$ is a $g$-frame for $H$ and $\phi$ is a computable map, $\{\Gamma_j^{*}(e_{jk}):j,k\in \NN\}$ is a computable frame in $H$. As $\phi^{*}$ is $(\delta_H, \delta_{(\sum\oplus H_i)_{l_{2}}})$ computable, $(\phi^{*}(f))_j$ is $\delta_{H_j}$ computable for given  computable $f\in H$ and $j\in\NN$, where $(\phi^{*}(f))_j$ denotes the  $j^{th}$ component of $\phi^{*}(f)$. This gives the $\delta_{l^2}$ computability of $(\langle (\phi^{*}(f))_j,e_{jk}\rangle)_k$ for given computable $f\in H$ and $j\in\NN$. Since
\begin{align*} \langle f,\Gamma_j^{*}e_{jk}\rangle=\langle f,\phi(E_{jk})\rangle=\langle \phi^{*}(f),E_{jk}\rangle=\langle (\phi^{*}(f))_j,e_{jk}\rangle \end{align*}
$(\langle f,\Gamma_j^{*}e_{jk}\rangle)_k$ is a computable element of $l^2$ for given computable $f\in H$ and $j\in\NN$. Now, the result follows from Theorem 3.7.
\end{proof}

For computable $g$-frames with computable frame operator, the existence of computable dual $g$-frames can be characterized by the following result. The result is a computable version of Theorem 3.4[2].
\begin{theorem}
Let $(\Lambda_i)$ be a computable $g$-frame for $H$ with respect to $\{H_i\}$ with computable frame operator $S$. Then, a $g$-frame $(\Gamma_i)$ with computable analysis operator is a computable dual $g$-frame of $(\Lambda_i)$  if and only if there exists a $(\delta_H, \delta_{(\sum\oplus H_i)_{l_2}})$ computable operator $\psi$ such that $T\psi=0, T$ being the synthesis operator of $g$-frame $(\Lambda_i)$.
\end{theorem}
\begin{proof}
Let $(\Lambda_i)$ be a computable $g$-frame with computable frame operator $S$ and $(\Gamma_i)$ be a computable dual $g$-frame of $(\Lambda_i)$ with computable analysis operator.\\
Define $\psi:H\rightarrow (\sum\oplus H_i)_{l^2}$ as\\
\[f\rightarrow\left(\begin{array}{ccccc}
 \Gamma_1f-\lambda_1S^{-1}f\\
 \Gamma_2f-\lambda_2S^{-1}f\\
 \vdots
 \end{array}\right),\]
 Then, $\psi$ is a  well defined bounded operator such that for given computable $f\in H$ and $i\in \NN, \Gamma_if-\Lambda_i S^{-1}f$ is $\delta_{H_i}$ computable by evaluation property.
 Since $(\Gamma_i)$ has computable analysis operator, $\sum_{i=1}^{\infty}\|\Gamma_i f\|^2$ is $\delta_{\RR}$-computable for given computable $f\in H$. The computability of operator $S^{-1}$ and inner product $\langle \cdot \rangle$ implies the computability of $\sum_{i=1}^{\infty}\|\Lambda_i S^{-1}f\|^2=\langle S^{-1}(SS^{-1}f),f\rangle=\langle S^{-1}f,f\rangle$ for given computable $f\in H$. Therefore, for given $f\in H$,
 \begin{align*}
 \sum_{i=1}^{\infty}\|\Gamma_if-\Lambda_iS^{-1}f\|^2=\sum_{i=1}^{\infty}\|\Gamma_i f\|^2+\sum_{i=1}^{\infty}\|\Lambda_iS^{-1}f\|^2+2\bigg(\sum_{i=1}^{\infty}\|\Gamma_i f\|^2\bigg)^{\frac{1}{2}}\bigg(\sum_{i=1}^{\infty}\|\Lambda_iS^{-1}f\|^2\bigg)^{\frac{1}{2}}
 \end{align*}
 is $\delta_{\RR}$-computable. Hence $\psi$ is $(\delta_H, \delta_{(\sum\oplus H_i)_{l_2}})$ computable operator such that $T\psi=0$.\\
 Conversely, let $\psi$ be a $(\delta_H, \delta_{(\sum\oplus H_i)_{l_2}})$ computable operator such that $T\psi=0.$\\
 Define $\Gamma_i:H\rightarrow H_i$ as $\Gamma_i(f)=\Lambda_iS^{-1}f+(\psi f)_i,$ $f\in H$, $i\in\NN$. Then, $(\Gamma_i)$ forms a computable dual $g$-frame to $(\Lambda_i)$. The associated analysis operator $T^{*}:H \rightarrow (\sum\oplus H_i)_{l_2}$ given by $f \rightarrow (\Gamma_i f)_i=(\Lambda_iS^{-1}f+(\psi f)_i)$ is computable as well as for given computable $f\in H$ and $i\in \NN$, $\Lambda_i S^{-1}f+(\psi f)_i$ is $\delta_{H_i}$ computable and
 \begin{align*}
 \sum_{i=1}^{\infty}\|\Gamma_if\|^2=\sum_{i=1}^{\infty}\|\Lambda_iS^{-1}f\|^2+\sum_{i=1}^{\infty}\|(\psi_i f)_i\|^2+2\bigg(\sum_{i=1}^{\infty}\|\Gamma_i f\|^2\bigg)^{\frac{1}{2}}\bigg(\sum_{i=1}^{\infty}\|(\psi f)_i\|^2\bigg)^{\frac{1}{2}}
  \end{align*}
  is $\delta_{\RR}$-computable because of the $\delta_{\RR}$-computability of $\langle S^{-1} f, f\rangle$ and $\|\psi f\|$ for given computable $f\in H$.
  \end{proof}
  
\mbox{}
\end{document}